\documentclass[a4paper,11pt]{amsart}

\usepackage{amsmath}
\usepackage{amsfonts} 
\usepackage{amssymb}
\usepackage{amsthm}
\usepackage[all]{xy}
\usepackage{tikz-cd}
\usepackage{enumerate} 
\usepackage[T1]{fontenc}
\usepackage{ytableau}

\theoremstyle{plain}
\newtheorem{thm}{Theorem}[section]
\newtheorem{lemma}{Lemma}[section]

\newtheorem{cor}{Corollary}[section]
\newtheorem{prop}{Proposition}[section]

\theoremstyle{remark}
\newtheorem*{rem}{Remark}

\theoremstyle{definition}
\newtheorem*{ex}{Example}

\newcommand{\R}{\mathbb{R}}
\newcommand{\N}{\mathbb{N}}
\newcommand{\C}{\mathbb{C}}

\newcommand{\pol}{\mathcal{P}}
\newcommand{\inv}{\mathcal{I}}
\newcommand{\har}{\mathcal{H}}
\newcommand{\kk}{\mathfrak{k}}
\newcommand{\g}{\mathfrak{g}}
\newcommand{\p}{\mathfrak{p}}
\newcommand{\q}{\mathfrak{q}}
\newcommand{\h}{\mathfrak{h}}
\newcommand{\U}{\mathcal{U}}

\newcommand{\A}{\mathcal{A}}

\newcommand{\lspan}{\operatorname{span}}
\newcommand{\Ker}{\operatorname{Ker}}
\newcommand{\End}{\operatorname{End}}

\begin{document}

\title[Separation of Variables in the Non-stable Range]{Separation of Variables for Scalar-valued Polynomials in the Non-stable Range}

\author{Daniel Be\v dat\v s}

\address{Charles University, Faculty of Mathematics and Physics, Mathematical Institute\\
Sokolovsk\'a 83, 186 75 Praha, Czech Republic}

\curraddr{ISTA (Institute of Science and Technology Austria), Am Campus 1, 3400 Klosterneuburg, Austria}

\email{daniel.bedats@ist.ac.at}

\thanks{The support of the grant GACR 20-114735 is gratefully acknowledged.}

\begin{abstract}
    Any complex-valued polynomial on $(\R^n)^k$ decomposes into an algebraic combination of $O(n)$-invariant polynomials and harmonic polynomials. This decomposition, separation of variables, is granted to be unique if $n \geq 2k-1$. We prove that the condition $n\geq 2k-1$ is not only sufficient, but also necessary for uniqueness of the separation. Moreover, we describe the structure of non-uniqueness of the separation in the boundary cases when $n = 2k-2$ and $n=2k-3$.

Formally, we study the kernel of a multiplication map $\phi$ carrying out separation of variables. We devise a general algorithmic procedure for describing $\Ker{\phi}$ in the restricted non-stable range $k \leq n < 2k-1$. In the full non-stable range $n < 2k-1$, we give formulas for highest weights of generators of the kernel as well as formulas for its Hilbert series. Using the developed methods, we obtain a list of highest weight vectors generating $\Ker{\phi}$.
\end{abstract}



\maketitle

\section{Introduction}
Every complex-valued polynomial $P$ in the Euclidean space $\R^n$ admits a \textit{unique} decomposition 
\[P = H_0 + r^2H_1 + \dots + r^{2m}H_m,\] where $r^2 = x_1^2 + \dots + x_n^2$ is the square of the Euclidean norm and each $H_i$ is a harmonic polynomial. That is, $\Delta(H_i) = 0$ for the Laplace operator
\[\Delta = \partial_{x_1}^2 + \dots + \partial_{x_n}^2.\]
Here, $x_1,\dots,x_n$ are coordinates on the Euclidean space $\R^n$.
This result, separation of variables, is a cornerstone of the classical theory of spherical harmonics. The underlying symmetry of the problem is given by the real orthogonal group of matrices preserving the Euclidean inner product. 

We can ask about a similar decomposition for polynomials on $(\R^n)^k$, i.e. in multiple vector variables. It is known that such decomposition always exists \cite{Ho1, Ho2} in any dimension $n$ and for any number of variables $k$. However, separation of variables for complex-valued polynomials in multiple vector variables need not be unique. This paper addresses the question of non-uniqueness, which is formally encoded as the kernel of a map $\phi$ defined in (\ref{phi}) below. Following convention \cite{G,Ho2}, we consider the complexified version of the problem with vector variables from $\C^n$ and the underlying symmetry given by the complex orthogonal group $O(n) = O(n,\C)$.

The theory of spherical harmonics can be formulated in a modern setting as an instance of Howe duality \cite{G, GW, Ho1, Ho2}. 
It says that the space $\pol = \pol((\C^n)^k)$ of complex-valued polynomials in $k$ vector variables admits a multiplicity-free decomposition into irreducible representations of the pair $(O(n),\mathfrak{sp}(2k,\C))$, where $\mathfrak{sp}(2k,\C)$ is the complex symplectic Lie algebra. We give an operator realization of $\mathfrak{sp}(2k,\C)$ and a precise formulation of this result in Section \ref{howeduality}. The Howe duality decomposition of $\pol$ involves unitarizable highest weight $\mathfrak{sp}(2k,\C)$-modules $L(\lambda)$. To understand non-uniqueness of separation of variables, we study these modules. The main tool for our exploration is a resolution of $L(\lambda)$ by generalized Verma modules introduced by Enright and Willenbring in \cite{EW}.

To begin studying decomposition of polynomials 
\[ P: (\C^n)^k \longrightarrow \C \]
into invariant and harmonic polynomials, let us identify $(\C^n)^k$ with the space $M_{k \times n} = M_{k \times n}(\C)$ of complex $k\times n$ matrices. Each of the $k$ vector variables of $(\C^n)^k$ then corresponds to a row in $M_{k \times n}$. For $1 \leq i \leq k$ and $1\leq j \leq n$, let $x_{ij}$ denote the coordinate function on $M_{k\times n}$ which maps a matrix to its entry at position $(i,j)$. Thus we are interested in the space 
\[ \pol = \pol(M_{k \times n}) = \C[x_{11},x_{12},\dots,x_{kn}]\] of $\C$-valued polynomials in vector variables $\mathbf{x}_i = (x_{i1},\dots,x_{in})$. When we mention $\pol$ and its special subsets $\inv$ and $\har$ defined below, we understand the dimension $n$ and the number of variables $k$ to be implicit.

We consider the complex orthogonal group $O(n) = O(n,\C)$ of $n\times n$ matrices preserving the symmetric bilinear form $(\cdot,\cdot)$ on $\C^n$ given by
\begin{equation} \label{symform}
    (u,v) = u_1v_n +  u_2v_{n-1} + \dots + u_n v_1, \quad u,v \in \C^n.
\end{equation}
We choose this antidiagonal form so that there is a maximal torus of $O(n)$ formed by its diagonal matrices. The group $O(n)$ acts on the space $\pol$ by
\[ g \cdot P(x) = P(xg), \quad g \in O(n), \quad P \in \pol, \]
where $xg$ is the product of the full matrix variable $x = (x_{ij})$ with $g \in O(n)$.

The ring  of $O(n)$-invariant polynomials in $\pol$ is defined as
\[\inv = \inv(M_{k \times n}) = \{ P\in \pol: g \cdot P = P \text{ for all } g \in O(n)\}.\] The First Fundamental Theorem of invariant theory \cite[p.~240]{GW} says that $\inv$ is generated by quadratic polynomials 
\begin{equation} \label{invariants}
    r_{ij} = (\mathbf{x}_i,\mathbf{x}_j) = x_{i1}x_{jn} + x_{i2}x_{j,n-1} + \dots + x_{in}x_{j1}, \quad 1 \leq i \leq j \leq k. 
\end{equation}
The space $\har = \har(M_{k \times n}) \subset \pol$ of harmonic polynomials is the common kernel of $O(n)$-invariant Laplace operators
\begin{equation} \label{laplacians}
    \Delta_{ij} = \partial_{x_{i1}}\partial_{x_{jn}} + \partial_{x_{i2}}\partial_{x_{j,n-1}} \dots + \partial_{x_{in}}\partial_{x_{j1}},  \quad 1 \leq i \leq j \leq k. 
\end{equation}
Separation of variables for complex-valued polynomials can be formally defined as the problem of deciding if the multiplication map $\phi$ given by
\begin{equation} \label{phi}
    \phi: \inv \otimes \har \rightarrow \pol, \quad I \otimes H \mapsto  IH,
\end{equation}
is a linear isomorphism. Here, $IH$ is the product of polynomials $I \in \inv$ and $H \in \har$. Throughout the paper, the symbol $\otimes$ denotes the tensor product $\otimes_\C$ of complex vector spaces. It is known that $\phi$ is surjective for any values of $n$ and $k$. The multiplication map is not always injective, though, and therefore we are interested in the kernel of $\phi$. This subspace of $\inv \otimes \har$ is the sole focus of this work, so \textit{`the kernel'} always means $\Ker{\phi}$ throughout the paper.

A classical result \cite{Ho2} saying that separation of variables is unique in the stable range $n\geq 2k$ was extended to the semistable range $n\geq 2k-1$ in \cite{Lav}. Building on this method, our paper provides partial results on the open problem of describing $\Ker{\phi}$ in the non-stable range $n < 2k-1$. We show that the kernel of $\phi$ is always non-zero in this range and provide its precise structural description in the boundary non-stable cases $n=2k-2$ and $n=2k-3$. More generally, we adopt an algorithmic procedure from \cite{EW} giving a resolution of $\Ker{\phi}$ by free $\inv$-modules in the restricted non-stable range $k \leq n < 2k-1$. We further introduce a method of finding highest weights of generators for the kernel, which we use in Section \ref{gen} to construct an explicit list of the generators in the full non-stable range $n < 2k-1$. To give a slightly different type of information on $\Ker{\phi}$, we provide a way of calculating its Hilbert series also applicable in the full non-stable range.

The paper is based on author's master thesis \cite{B} defended at Charles University, Prague.

\section{Methods for understanding the kernel}

\subsection{Howe duality decompositions} \label{howeduality}
The starting point of our exploration is the Howe duality $(O(n),\mathfrak{sp}(2k,\C))$. Let  $\g = \p^- \oplus \kk \oplus \p^+$ denote the complex Lie algebra formed by the following $O(n)$-invariant differential operators on $\pol$: 
\begin{align*}
    \p^- &= \lspan_\C \{r_{ij}: 1 \leq i \leq j \leq k\}, \\
     \kk &= \lspan_\C \{h_{ij}: 1 \leq i,j \leq k \},\\
    \p^+ &= \lspan_\C \{\Delta_{ij}: 1 \leq i \leq j \leq k \}.
\end{align*} 
The multiplication operators $r_{ij}$ and the Laplace operators $\Delta_{ij}$ are defined in (\ref{invariants}) and (\ref{laplacians}) above, while the mixed operators $h_{ij}$ are given by 
\[ h_{ij} = x_{i1} \partial_{x_{j1}} + \dots + x_{in}\partial_{x_{jn}} + \frac{n}{2} \delta_{ij}, \quad 1 \leq i,j \leq k. \]
Howe duality says that there is a multiplicity-free decomposition of $\pol$ under the joint action of the pair $(O(n),\g)$, see below. Before formulating the decomposition, we develop necessary conventions regarding the structure of $\g$.

The Lie algebra $\g = \p^- \oplus \kk \oplus \p^+$ is an operator realization of the complex symplectic Lie algebra $\mathfrak{sp}(2k) = \mathfrak{sp}(2k,\C)$ with its subalgebra $\kk$ forming an instance of $\mathfrak{gl}(k) = \mathfrak{gl}(k,\C)$. A Cartan subalgebra common for both $\g$ and $\kk$ is 
\[\h= \lspan{\{h_{11},\dots,h_{kk}\}}.\] 
We choose the basis $-h_{11},\dots,-h_{kk}$ of $\h$ and denote by \[\varepsilon_1,\dots,\varepsilon_k \in \h^*, \quad \varepsilon_i: -h_{jj}\mapsto \delta_{ij},\] its dual basis of $\h^*$. Then $\varepsilon_i + \varepsilon_j$ is a root of $(\g,\h)$ corresponding to the root vector $\Delta_{ij}$ and $-\varepsilon_i-\varepsilon_j$ is a root with root vector $r_{ij}$. 
Moreover, $\varepsilon_i - \varepsilon_j$ is a root with root vector $h_{ji}$ when $i \neq j$. Let $\Delta$ denote the resulting root system of $(\g,\h)$. The subsystem $\Delta_{\kk} = \{ \pm(\varepsilon_i - \varepsilon_j): i < j\}$ forms a root system of $(\kk,\h)$. The elements of $\Delta_{\kk}$ are referred to as compact roots. 

Take $\varepsilon_1 - \varepsilon_{2},\dots,\varepsilon_{k-1}-\varepsilon_k$ and $2\varepsilon_k$ to be the simple roots of $\Delta$ and denote the resulting positive root system by $\Delta^+$. Let $\Delta_{\kk}^+ = \Delta^+ \cap \Delta_\kk$ be the positive compact roots. Then we can speak of highest weights for $\g$ and for $\kk$. Throughout the text, we identify a weight $\lambda \in \h^*$ with its coordinates $(\lambda_1,\dots,\lambda_k)$ with respect to $\varepsilon_1,\dots,\varepsilon_k$.

Now we give decompositions of $\pol = \pol(M_{k \times n})$ and $\har \subset \pol$ under the actions of the dual pairs $(O(n),\mathfrak{sp}(2k))$ and $(O(n),\mathfrak{gl}(k))$, respectively. First, we briefly recall some essential facts about irreducible representations of $O(n)$. A \textit{partition} is a finite weakly decreasing sequence $\lambda = (\lambda_1,\dots,\lambda_m)$ of non-negative integers. We identify a partition $\lambda$ with its \textit{Young diagram}, which is a graphic representation of $\lambda$ as a grid of left-justified rows with $\lambda_i$ boxes in the $i$-th row. The \textit{depth} of $\lambda = (\lambda_1,\dots,\lambda_m)$ is the number of rows in its Young diagram or, in other words, the largest integer $d$ such that $\lambda_d > 0$. For example,

\[ \lambda = (4,3,1,0,0) = \ydiagram{4,3,1} \]
has depth equal to 3.

 Highest weight theory for the special orthogonal group $SO(n)$ gives a classification of its irreducible representations in terms of partitions. As a consequence, the finite-dimensional irreducible representations of $O(n)$ are classified by the set $\Sigma_n$ of Young diagrams having at most $n$ boxes in the first two columns \cite{G, GW, Ho2}. A representation labeled by $\sigma \in \Sigma_n$ is here denoted by $E_\sigma$.

Let $\Sigma_{n,k} \subset \Sigma_n$ denote the set of Young diagrams from $\Sigma_n$ of depth at most $k$. Then $\Sigma_{n,k}$ is the set of labels for irreducible $O(n)$-representations occurring in $\pol(M_{k \times n})$. For more details, see \cite[Lecture 10]{G}. Howe duality \cite{G, Ho1} gives a nice decomposition of $\pol = \pol(M_{k \times n})$ under the commuting actions of $O(n)$ and its dual partner $\mathfrak{sp}(2k)$. Each irreducible $O(n)$-module $E_\sigma$ occurring in $\pol$ has an infinite multiplicity, which is resolved by tensoring with a unique infinite-dimensional $\mathfrak{sp}(2k)$-module of suitable type. More precisely, there is a multiplicity-free decomposition
\[ \pol =  \bigoplus_{\sigma \in \Sigma_{n,k}} E_\sigma \otimes L(\sigma^\sharp)  \]
under the joint action of the pair $(O(n),\mathfrak{sp}(2k))$,
where $L(\lambda)$ is the irreducible highest weight $\mathfrak{sp}(2k)$-module of highest weight $\lambda \in \h^*$. The transformation $\sigma \mapsto \sigma^\sharp$ maps $\Sigma_{n,k}$ bijectively to the set of highest weights for irreducible $\mathfrak{sp}(2k)$-modules occurring in $\pol$. In coordinates with respect to the basis $\varepsilon_1,\dots,\varepsilon_k$, the highest weight $\sigma^\sharp \in \h^*$ has the expression
\[ \sigma^\sharp = (-\sigma_k -\frac{n}{2}, \dots, -\sigma_1 - \frac{n}{2}).\]

The Howe duality decomposition of $\pol$ descends to a decomposition of $\har \subset \pol$ with the action of $\g = \mathfrak{sp}(2k)$ restricted to the action of $\kk = \mathfrak{gl}(k)$. First of all, we note that the set $\Sigma_{n,k}$ also classifies the irreducible $O(n)$-representations occurring in $\har$. In other words, the spaces $\pol$ and $\har$ have the same $O(n)$-spectrum. Moreover, there is  a multiplicity-free decomposition \cite[Lecture 9]{G}
\[ \har = \bigoplus_{\sigma \in \Sigma_{n,k}} E_\sigma \otimes F_{\sigma^\sharp} \]
under the joint action of the pair $(O(n),\mathfrak{gl}(k))$,
where $F_\lambda$ denotes the finite-dimensional irreducible $\mathfrak{gl}(k)$-module of highest weight $\lambda \in \h^*$.

In addition to sharing the $O(n)$-spectrum, the decompositions of $\pol$ and of $\har$ are connected in the following natural sense. Let us take $\sigma \in \Sigma_{n,k}$ and choose a realization of $F_{\sigma^\sharp}$ as a subspace of $\har$. Then it holds that the highest weight $\mathfrak{sp}(2k)$-module $L(\sigma^\sharp)$ is realized in $\pol$ as $L(\sigma^\sharp) = \inv F_{\sigma^\sharp}$, the span of all products of $\inv$ with $F_{\sigma^\sharp}$. Thus Howe duality provides a concrete construction of abstract highest weight modules $L(\lambda)$ in terms of polynomials.

Note that for each $\sigma \in \Sigma_{n,k}$, the module $L(\sigma^\sharp)$ is not only an irreducible highest weight module as usually stated in results concerning Howe duality, but also unitarizable. This is thanks to the Fischer inner product $\langle \cdot | \cdot \rangle$ on $\pol$ given by $\langle P|Q \rangle = (\partial(P)Q^*)(0)$, where $Q^*(x) = \overline{Q(\overline{x})}$ and $\overline{z}$ denotes the complex conjugate of $z$. 

\begin{rem}
    There is an easy proof of surjectivity of the multiplication map $\phi: \inv \otimes \har \rightarrow \pol$ using orthogonality with respect to $\langle \cdot | \cdot \rangle$. If we denote by $\inv_+ \subset \inv$ the space of invariants of positive degree, then there is an orthogonal decomposition $\pol = \har \oplus \inv_+\pol$, see \cite[Theorem 9.2]{G}. Surjectivity of $\phi$ follows.
\end{rem}

\subsection{Generalized Verma modules} \label{GVM}
The Lie algebras $\g = \mathfrak{sp}(2k)$ and $\kk = \mathfrak{gl}(k)$ constitute an example of a Hermitian symmetric pair, which allows for application of rich theory to our exploration. In particular, if $(\g,\kk)$ is a Hermitian symmetric pair and $\q$ is a maximal parabolic subalgebra of $\g$ whose Levi subalgebra is $\kk$, then strong results describing the unitarizable highest weight modules $L(\lambda)$ are available. To use these results, we have to define generalized Verma modules $N(\lambda)$ associated to $\q$. For an overview of the theory, we refer to \cite{Hum}. 

Consider the maximal parabolic subalgebra $\q = \kk \oplus \p^+$ of $\g$, whose Levi subalgebra is indeed $\kk$. Let $\lambda \in \h^*$ be a $\kk$-dominant integral weight and let $F_\lambda$ be an irreducible $\kk$-module of highest weight $\lambda$ and let $\p^+$ act on $F_\lambda$ trivially, i.e. by $\p^+ \cdot F_\lambda = 0$. Then $F_\lambda$ is a $\q$-module and we define the generalized Verma module of highest weight $\lambda$ with respect to the parabolic subalgebra $\q$ as the induced $\g$-module
\[ N(\lambda) = \U(\g) \otimes_{\U(\q)} F_\lambda, \]
where $\U(\mathfrak{a})$ is the universal enveloping algebra of a Lie algebra $\mathfrak{a}$.
As a left $S(\p^-)$-module, $N(\lambda)$ is isomorphic to $S(\p^-)\otimes F_\lambda$ by the Poincar\'e-Birkhoff-Witt (PBW) Theorem \cite[2.1.11]{D} thanks to the decomposition $\g = \p^- \oplus \q$ and the commutativity of $\p^-$. Moreover, $N(\lambda)$ contains a unique proper maximal submodule and thus admits a unique irreducible quotient, the irreducible highest weight module $L(\lambda)$. Note that in \cite{DES} and \cite{EW}, symbols $N(\lambda + \rho)$ and $L(\lambda+\rho)$ are used in place of $N(\lambda)$ and $L(\lambda)$. 

In Section \ref{GVMforA}, we introduce similar modules $N_\A(\lambda)$  associated to the algebra $\A$ of $O(n)$-invariant differential operators on $\pol(M_{k\times n})$, which forms a natural quotient of $\U(\g)$.

\subsection{Invariance of $\phi$ }

Let us describe $\mathfrak{sp}(2k)$-module structures on $\inv \otimes \har$ and on $\pol = \pol(M_{k\times n})$ and then observe that $\phi: \inv \otimes \har \rightarrow \pol$ intertwines them. The action on $\pol$ is straightforward; the symplectic Lie algebra $\g = \mathfrak{sp}(2k)$ acts on $\pol$ by the differential operators of which it consists. To endow $\inv \otimes \har$ with a $\g$-module structure, we have to do a bit more work.

Let $\End(\pol)$ denote the algebra of linear endomorphisms of $\pol$. Consider the subalgebra $\A \subset \End(\pol)$ generated by operators $r_{ij},h_{ij},\Delta_{ij}$ spanning the Lie algebra $\g$. Further let $\mathcal{D}(\pol)$ denote the Weyl algebra on $\pol$, i.e. the algebra of polynomial-coefficient differential operators on $\pol$. Then $\A = \mathcal{D}(\pol)^{O(n)}$ is exactly the subalgebra of $\mathcal{D}(\pol)$ consisting of $O(n)$-invariant differential operators, see e.g. \cite[Theorem 7]{Ho1}. The inclusion of $\g$ in $\End(\pol)$ induces an algebra map $\psi: \U(\g) \rightarrow \End(\pol)$ whose image is the algebra $\A$. Thus $\A$ is the quotient of $\U(\g)$ by the two-sided ideal $\Ker{\psi}$.  For details on the structure of $\Ker{\psi}$, we refer to \cite[Section 5]{MW}. We can think of $\A$ as an enveloping algebra of $\g$ within the Weyl algebra $\mathcal{D}(\pol)$. The key takeaway for us is that the algebra $\A$ admits a natural left $\U(\g)$-module structure given by composition of operators in $\A$:
\[X \cdot A = \psi(X)A, \quad X \in \U(\g), \quad A \in \A.\] 
Similarly, $\A$ forms a right $\U(\g)$-module with composition from the right. Recall the parabolic subalgebra $\q = \kk \oplus \p^+$ of $\g$ spanned by operators $\Delta_{ij}$ and $h_{ij}$. Then the space $\har$ of harmonics forms a left $\U(\q)$-module, while $\A$ forms a right $\U(\q)$-module since $\U(\q)\subset \U(\g)$. Therefore we can form the tensor product $\A \otimes_{\U(\q)} \har$, which is naturally a left $\U(\g)$-module.
 
Observe that $\inv \subset \A$ once we identify an invariant $I \in \inv$  with the multiplication operator $I: P \mapsto IP$ on $\pol$. This inclusion induces an embedding $i: \inv \otimes \har \rightarrow \A \otimes \har$. On the other hand, there is a natural surjection of tensor products $q: \A \otimes \har \rightarrow \A \otimes_{\U(\q)} \har$. The composition
\[ \varphi: \inv \otimes \har \rightarrow \A \otimes_{\U(\q)} \har, \quad \varphi = q \circ i,\]
is an isomorphism of left $\inv$-modules by the PBW Theorem. This endows the space $\inv \otimes \har$ with the left $\U(\g)$-module structure of $\A \otimes_{\U(\q)} \har$, which is in agreement with the natural left $\inv$-module structure on $\inv \otimes \har$.

Finally, observe that the map $\phi: \inv \otimes \har \rightarrow \pol$ intertwines the actions of $\g$. Indeed, the composition of two operators $X \in \g$ and $I \in \inv$ on $\pol$ is the same as their product in $\A$ and hence    
\[ X \cdot (IH) = (XI)H =\phi ((XI) \otimes_{\U(\q)} H), \quad X \in \g, \quad I \in \inv, \quad H \in \har.\]
Note that $\phi$ also intertwines the natural tensor product action of $O(n)$ on $\inv \otimes \har$ with the action of $O(n)$ on $\pol$. Of course, the action on $\inv \otimes \har$ is trivial on the factors from $\inv$ and hence given by $g \cdot (I \otimes H) = I \otimes (g\cdot H)$ for $g \in O(n), 
\: I \in \inv, \: H \in \har$.

\subsection{Generalized Verma modules for $\A$} \label{GVMforA}

When we compare the definition of a generalized Verma module $N(\lambda)$ and the module structure on $\inv \otimes \har$, there is an important similarity; both are a result of tensoring over the parabolic algebra $\U(\q)$ with a $\q$-module on which $\p^+$ acts trivially. Now we make this connection more formal. As in Section \ref{GVM}, let $\lambda \in \h^*$ be a $\kk$-dominant integral weight and take an irreducible $\kk$-module $F_\lambda$ on which $\p^+$ acts trivially. By definition and by Howe duality, we know that $\har$ splits into such $\kk$-modules.

Define the generalized Verma module for $\A$ with respect to the parabolic subalgebra $\q$ as the induced $\g$-module
\[ N_\A(\lambda) = \A \otimes_{\U(\q)} F_\lambda. \] 
Then $N_\A(\lambda)$ is a quotient of $N(\lambda)$ since $\A$ is a quotient of $\U(\g)$. Moreover, there is a left $\inv$-module isomorphism $N_\A(\lambda) \cong \inv \otimes F_\lambda$ by the PBW Theorem. Modules $N_\A(\lambda)$ can be viewed as parabolic versions of Verma modules for algebras of invariant differential operators studied in \cite{S}. Note that our situation provides concrete realizations of these abstract modules as tensor products $\inv\otimes F_\lambda$ of two spaces of special polynomials. 

We have defined the $\g$-module structure on $\inv \otimes \har$ in agreement with the definition of modules $N_\A(\lambda)$. As a consequence, Howe duality for $\har$ gives the $(O(n),\mathfrak{sp}(2k))$-module decomposition
\begin{equation} \label{invhar}
     \inv \otimes \har = \bigoplus_{\sigma \in \Sigma_{n,k}} E_\sigma \otimes N_\A(\sigma^\sharp).
\end{equation}

\subsection{Decomposition of $\Ker{\phi}$}
Let us decompose the kernel of the multiplication map $\phi$ into isotypic components for $O(n)$. Take $\sigma \in \Sigma_{n,k}$ and choose a joint highest weight vector $0 \neq H_\sigma \in \har$ for the pair $(SO(n), \mathfrak{gl}(k))$ of weight $\sigma^\sharp$ with respect to $\mathfrak{gl}(k)$ which generates an instance of $E_\sigma$ in $\pol$ under the action of $O(n)$. Then $F_{\sigma^\sharp} \subset \har$ is realized in the space of harmonics as the $\mathfrak{gl}(k)$-module generated by $H_\sigma$ and we can consider the restricted multiplication map
\[ \phi_\sigma: N_\A(\sigma^\sharp) = \inv \otimes F_{\sigma^\sharp} \rightarrow \inv F_{\sigma^\sharp} = L(\sigma^\sharp), \quad I\otimes H \mapsto IH.\]

Having chosen realizations of both $E_\sigma$ and $F_{\sigma^\sharp}$ as subsets of $\har$ generated by $H_\sigma$ under the actions of $O(n)$ and $\mathfrak{gl}(k)$, respectively, we can identify the $\sigma$-isotypic component $\har_\sigma$ in $\har$ with $E_\sigma \otimes F_{\sigma^\sharp}$. Consider the restriction $\phi|_{\inv \otimes \har_\sigma}: \inv \otimes \har_\sigma \rightarrow \pol$ of $\phi$ to $\inv \otimes \har_\sigma$, which is quite different from $\phi_\sigma$ since it involves $E_\sigma$ in its domain. The identification $\har_\sigma \cong E_\sigma \otimes F_{\sigma^\sharp}$ allows us to express $\phi|_{\inv \otimes \har_\sigma}$ as 
\[id \otimes \phi_\sigma: E_\sigma \otimes N_\A(\sigma^\sharp) \rightarrow E_\sigma \otimes L(\sigma^\sharp), \]
where $id: E_\sigma \rightarrow E_\sigma$ is the identity. As a consequence, we obtain the decomposition
\[ \Ker{\phi} =  \bigoplus_{\sigma \in \Sigma_{n,k}} \Ker{(\phi|_{\inv \otimes \har_\sigma})} = \bigoplus_{\sigma \in \Sigma_{n,k}} E_\sigma \otimes \Ker{\phi_\sigma}\] under the action of the pair $(O(n),\mathfrak{sp}(2k))$, which reduces the problem of describing $\Ker{\phi}$ to understanding each $\Ker{\phi_\sigma}$. Of course, it may happen that for some $\sigma \in \Sigma_{n,k}$, the kernel of $\phi_\sigma$ is zero.

Before studying the restricted multiplication map $\phi_\sigma$, we would like to decide when the generalized Verma modules $N(\lambda)$ and $N_\A(\lambda)$ coincide. If $n \geq 2k$, then it holds \cite[Theorem 5.2]{MW} that $\U(\g) = \A$ and the equality $N(\lambda) = N_\A(\lambda)$ is automatic. However, even if $\U(\g)$ and $\A$ are distinct, the respective generalized Verma modules can be the same. Indeed, the ring $\inv$ of invariants is generally a quotient of $S(\p^-)$ and equal to $S(\p^-)$ if and only if $n \geq k$. The precise structure of $\inv$ is the content of the Second Fundamental Theorem of invariant theory (SFT), see \cite[p.~561]{GW}. By definition, the condition $\inv = S(\p^-)$ is sufficient for $N(\lambda) = N_\A(\lambda)$ to hold for each $\lambda$. So if $n\geq k$, the map \[\phi_\sigma: \inv \otimes F_{\sigma^\sharp} \rightarrow \inv F_{\sigma^\sharp}\] is an instance of the natural projection of the generalized Verma module $N(\sigma^\sharp)$ onto $L(\sigma^\sharp)$. Therefore $\Ker{\phi_\sigma}$ is isomorphic to the maximal submodule in $N(\sigma^\sharp)$ in this case, which is particularly useful since the modules $N(\lambda)$ are generally understood better than the modules $N_\A(\lambda)$.

In general, the map $\phi_\sigma$ is a surjection of $N_\A(\sigma^\sharp)$ onto $L(\sigma^\sharp)$. It fits into the commutative diagram of $\g$-module maps below, where $\pi_\sigma$ is the surjection of $ N(\sigma^\sharp)$ onto $N_\A(\sigma^\sharp)$ induced by the natural projection $\U(\g)\rightarrow \A$.
\[ \begin{tikzcd}
 N(\sigma^\sharp)  \arrow{r}{\pi_\sigma} \arrow[swap]{dr}{\widetilde{\phi_\sigma}} & N_\A(\sigma^\sharp) \arrow{d}{\phi_\sigma} \\
     & L(\sigma^\sharp)
  \end{tikzcd} \]
   Instead of $\phi_\sigma$, the composition $\widetilde{\phi_\sigma} = \phi_\sigma \circ \pi_\sigma$ fits the setting as the projection $N(\sigma^\sharp) \rightarrow L(\sigma^\sharp)$. Thus $\Ker{\phi_\sigma} = \pi_\sigma( \Ker{\widetilde{\phi_\sigma})}$ is generally isomorphic to the quotient of the maximal submodule in $N(\sigma^\sharp)$ by the kernel of $\pi_\sigma$. Note that it might happen that $\pi_\sigma$ is an isomorphism, but $\phi_\sigma$ is not, and vice versa.

   \subsection{Resolution of unitarizable highest weight modules}
   
   In Theorem \ref{resolution} below, we introduce a resolution of $L(\sigma^\sharp)$ in terms of generalized Verma modules adapted from Enright and Willenbring \cite[p. 340]{EW}, which is our main tool for understanding the kernel of $\phi_\sigma$. To be able to formulate the theorem, we need to develop some more notation.

   Let $\Delta_n^+$ denote the set of positive non-compact roots so that $\Delta^+$ is the disjoint union $\Delta_{\kk}^+ \cup \Delta_n^+$. Let $W$ be the Weyl group of $\Delta$. Given a root $\alpha \in \Delta$, its corresponding root reflection is denoted by $s_\alpha  \in W$. A coroot for $\alpha$ is the element $\alpha^\vee \in \h$ satisfying \[\mu(\alpha^\vee) = \frac{2(\alpha,\mu)}{(\alpha,\alpha)}, \quad \mu \in \h^*,\]
   where $(\cdot,\cdot)$ is the Killing form on $\g$. Let $\rho$ denote the half-sum of positive roots in $\Delta$.

   Let $\lambda \in \h^*$ be a $\kk$-dominant integral weight. Define the set $\Psi_\lambda$ of singularities of $\lambda + \rho$ as the set of roots orthogonal to $\lambda +\rho$: 
   \[\Psi_\lambda = \{ \alpha \in \Delta : (\alpha, \lambda + \rho) = 0 \}. \]
    Let $W_\lambda$ be the subgroup of $W$ generated by the root reflections $s_\alpha$ with $\alpha \in \Delta_n^+$ satisfying the following conditions:
    \begin{enumerate}[(i)]
        \item $ (\lambda + \rho, \alpha^\vee) \in \N^+ = \{1,2,3,\dots\}$,
        \item $\alpha$ is orthogonal to $\Psi_\lambda$,
        \item if $\Psi_\lambda$ contains a long root, then $\alpha$ is short. 
    \end{enumerate}
    From now on, we refer to this list simply as conditions (i)-(iii). Let $\Delta_\lambda$ denote the set of roots $\alpha \in \Delta$ for which $s_\alpha \in W_\lambda$. Then $\Delta_\lambda$ is an abstract root system. Set $\Delta_\lambda^+ = \Delta_\lambda \cap \Delta^+$ and $\Delta_{\lambda,\kk} = \Delta_\lambda \cap \Delta_\kk$ with positive roots $\Delta_{\lambda,\kk}^+ = \Delta_{\lambda,\kk} \cap \Delta^+$. We further define $W^\kk_\lambda = \{ w \in W_\lambda : \Delta_{\lambda, \kk}^+ \subset w \Delta_\lambda^+\}$ and
    \[ W_\lambda^{\kk,i} = \{ w \in W_\lambda^\kk : \text{card}(w\Delta_\lambda^+ \cap - \Delta_\lambda^+) = i\}.\]
    Finally, if $\mu \in \h^*$ is any $\kk$-integral weight, let $\mu^+$ denote the unique $\kk$-dominant integral weight in its orbit of the Weyl group for $\Delta_\kk$.
    
    Now we are ready to describe the resolution of Enright and Willenbring, which can be viewed as an analogue of the Bernstein-Gelfand-Gelfand (BGG) resolution for unitarizable highest weight modules in terms of generalized Verma modules.

\begin{thm} \cite[p.~340]{EW} \label{resolution}
    Let $\lambda \in \h^*$ be a $\kk$-dominant integral weight. Let $L(\lambda)$ be a unitarizable highest weight $\mathfrak{sp}(2k)$-module of highest weight $\lambda$. Then $L(\lambda)$ admits a resolution
    \[ 0 \rightarrow Z_{r_\lambda}^\lambda \rightarrow \dots \rightarrow Z_1^\lambda \rightarrow N(\lambda) \rightarrow L(\lambda) \rightarrow 0,\]
    where $r_\lambda = \textnormal{card}(\Delta_\lambda^+ \cap \Delta_n^+)$ and $Z_i^\lambda = \sum_{w \in W_\lambda^{\kk,i}} N((w(\lambda+\rho))^+ - \rho)$ for $1 \leq i \leq r_\lambda$.
\end{thm}

The application of this result to our problem is straightforward. We can calculate the resolution of  $L(\sigma^\sharp)$ for each $\sigma \in \Sigma_{n,k}$. If $n\geq k$ and $\lambda = \sigma^\sharp$, we obtain a resolution
\[ 0 \rightarrow Z_{r_\lambda}^\lambda \rightarrow \dots \rightarrow Z_1^\lambda \rightarrow \Ker{\phi_\sigma} \rightarrow 0\]
of $\Ker{\phi_\sigma}$ by free $\inv$-modules $Z_i^\lambda$. Taking a direct sum over $\Sigma_{n,k}$, we obtain a structural description of the kernel of $\phi$. Even if $n < k$, the resolution of $L(\sigma^\sharp)$ provides useful information and can be used for example to compute the Hilbert series of the kernel, see Section \ref{hilbert}.

\subsection{Highest weights of generators of $\Ker{\phi}$} \label{generators}
In the general case, we wish to at least find a set of generators for the kernel of $\phi$. This information is stored in the first resolving module $Z_1^\lambda$. Indeed, for each $\sigma \in \Sigma_{n,k}$, the kernel of $\phi_\sigma: \inv \otimes F_{\sigma^\sharp} \rightarrow L(\sigma^\sharp)$ is a quotient of the kernel of $\widetilde{\phi_\sigma}: N(\sigma^\sharp) \rightarrow L(\sigma^\sharp)$, which is in turn isomorphic to a quotient of $Z_1^\lambda$. There is a quick way \cite[p.~24]{DES} of calculating $Z_1^\lambda$ which does not require finding the full root systems $\Delta_\lambda$ and $\Delta_{\lambda,\kk}$.

Fix a $\kk$-dominant integral weight $\lambda \in \h^*$ and denote by $\Gamma_\lambda$ the set of non-compact positive roots $\alpha \in \Delta_n^+$ satisfying conditions (i)-(iii). We consider the usual partial order on $\Delta^+$ given by $\alpha < \beta$ if and only if $\beta- \alpha \in \Delta^+$.

\begin{prop} \label{propgen}
    Let $L(\lambda)$ be a unitarizable highest weight module. Then the set $\Gamma_\lambda$ is either empty, or it has a unique minimum $\gamma$ such that $Z_1^\lambda = N((s_{\gamma}(\lambda + \rho))^+-\rho)$.
\end{prop}

So if we manage to find a non-zero element $v \in \Ker{\phi_\sigma}$ of weight $\lambda^\prime = (s_{\gamma}(\lambda + \rho))^+-\rho$, we can conclude that $v$ generates the kernel of $\phi_\sigma$ as a $\g$-module. In Section \ref{gen}, we use this approach to find generators of the kernel of $\phi$ for any number of variables in any dimension. Moreover, the action of the Levi subalgebra $\kk=\mathfrak{gl}(k)$ generates an instance of $F_{\lambda^\prime}$ in $\Ker{\phi_\sigma}$ and so the kernel of $\phi_\sigma$ is an $\inv$-module quotient of the free module $N_\A(\lambda^\prime)=\inv \otimes F_{\lambda^\prime}$. 

\subsection{Hilbert series of $\Ker{\phi}$} \label{hilbert}
In addition to the generating set, we can compute the Hilbert series of the kernel in the general cases when explicit structure through Theorem \ref{resolution} is unavailable.

All modules encountered in this paper inherit a grading by non-negative integers from the natural grading of the symmetric algebra $S(\p^-)$. More precisely, any generalized Verma module $N(\lambda) = S(\p^-) \otimes F_\lambda$ admits the grading $N(\lambda)^i = S(\p^-)^i \otimes F_\lambda$, which induces a grading of the irreducible quotient $L(\lambda)$. Moreover, the ring $\inv$ of invariants inherits a grading from $S(\p^-)$ and for each $\sigma \in \Sigma_{n,k}$, this induces a grading on $\inv \otimes F_{\sigma^\sharp}$. Therefore $\Ker{\phi_\sigma}$ admits a grading as its submodule. Let us remark that the grading of harmonic polynomials given by their homogeneous degree does not play a role here; regardless of $\sigma$, the space $F_{\sigma^\sharp} = \C \otimes F_{\sigma^\sharp}$ always forms the $0$-th graded component of $\inv \otimes F_{\sigma^\sharp}$. Note that for all modules mentioned here, each of the graded components is a finite-dimensional vector space. 

Given an $S(\p^-)$-module $M=\bigoplus_i M^i$ graded by non-negative integers, we define its Hilbert series by
\[H_M(q) = \sum_{i=0}^\infty \dim_\C(M^i) q^i.\]
Hilbert series of generalized Verma modules are easily computable as each $N(\lambda)\cong S(\p^-)\otimes F_\lambda$ is free over $S(\p^-)$. Theorem \ref{resolution} gives a resolution of $L = L(\lambda)$ by free $S(\p^-)$-modules, which yields a straightforward way of computing $H_L(q)$. Note that \cite{EW} provides an explicit formula for $H_L(q)$ with $L = L(\lambda)$ under the assumption that $\lambda$ is \textit{quasi-dominant}, but this is not satisfied in most of our cases of interest.

We are ultimately interested in the Hilbert series of $\Ker{\phi}$. It can be easily derived from what we have indicated above. Fix $\sigma \in \Sigma_{n,k}$ and consider the map $\phi_\sigma: \inv \otimes F_{\sigma^\sharp} \rightarrow L(\sigma^\sharp)$. The Hilbert series of $\inv \otimes F_{\sigma^\sharp}$ can be obtained directly from the Hilbert series of $\inv$. The SFT \cite[p.~561]{GW} gives the precise structure of $\inv$ as a quotient of $S(\p^-)$ and can be used to calculate $H_\inv (q)$. So if we set $K = \Ker{\phi_\sigma}$ and $L = L(\sigma^\sharp)$, we obtain the formula
\begin{equation} \label{hsformula}
    H_K(q) = \dim(F_{\sigma^\sharp})H_\inv(q) - H_L(q).
\end{equation} 

\subsection{Example: Hilbert series in dimension two} \label{hilbertex}
Let us illustrate the calculation of Hilbert series in the case $n=2, \: k=3$. First of all, we compute $H_\inv(q)$ using the SFT, which says that in this case $\inv$ is a quotient of $S(\p^-)$ by the principal ideal generated by the determinant of the symmetric $3\times 3$ matrix $(r_{ij})$. The Hilbert series of this principal ideal equals the Hilbert series of $S(\p^-)$ shifted by the degree of the determinant. Therefore 
\[H_\inv(q) = H_{S(\p^-)}(q) - q^3H_{S(\p^-)}(q) = \frac{1+q+q^2}{(1-q)^5}.\]
Note that the SFT is hidden implicitly in the resolution of Theorem \ref{resolution} as well, which is another way of obtaining $H_\inv(q)$. Indeed, consider the trivial label $\sigma = (0)$, then $\inv = L(\sigma^\sharp)$ and the resolution of $L(\sigma^\sharp)$ describes $\inv$ as the quotient of $N(\sigma^\sharp) = S(\p^-)\otimes F_{\sigma^\sharp}$. To calculate dimensions of the $\mathfrak{gl}(3)$-modules $F_\lambda$, we use the Weyl dimension formula \cite[p.~337]{GW}.

The strategy is straightforward; for each 
\[\sigma \in \Sigma_{2,3} = \Sigma_2 =  \{(d): d \geq 0\} \cup \{(1,1)\},\] we compute the resolution of $L(\sigma^\sharp)$ by generalized Verma modules, from which we easily obtain its Hilbert series. Having calculated $H_\inv(q)$, we use formula (\ref{hsformula}) to finish the task at hand. Note that the exponents of $q$ in the Hilbert series of $L(\lambda)$ are governed by the \textit{level of reduction} of $N(\lambda)$, which is the positive integer capturing the lowest degree of invariants contributing to $\Ker{\phi_\sigma}$. In our setting, it can be also computed as the contraction $\langle\lambda+\rho, \gamma^\vee \rangle$, where $\gamma^\vee$ is the coroot of $\gamma = \min{\Gamma_\lambda}$ from Proposition \ref{propgen}. See \cite[pp.~24-25]{DES} for more details. 

If $\sigma = (1,1)$, the resolution of $L=L(\sigma^\sharp)$ from Theorem \ref{resolution} is a short exact sequence
\[ 0 \rightarrow N(-2,-2,-3) \rightarrow N(-1,-2,-2) \rightarrow L \rightarrow 0. \]
The level of reduction for $L$ is one and equals the exponent of $q$ by which we multiply the Hilbert series for $N(-2,-2,-3)$ in the calculation
\[H_L(q) = H_{N(-1,-2,-2)} - qH_{N(-2,-2,-3)} = \frac{3}{(1-q)^5}.\]
Letting $K=\Ker{\phi_\sigma}$, we apply formula (\ref{hsformula}) and conclude that
\[H_K(q) = 3q\frac{1+q}{(1-q)^5}. \]

For $\sigma = (1)$ and $\sigma = (2)$, the resolution of $L=L(\sigma^\sharp)$ is again a short exact sequence. Hence the maximal submodule of $N=N(\sigma^\sharp)$ is isomorphic to $Z_1^{\sigma^\sharp}$, which is the generalized Verma module $N(-2,-3,-3)$ for $\sigma = (1)$ and $N(-1,-3,-3)$ for $\sigma = (2)$. Therefore the Hilbert series of $K = \Ker{\phi_\sigma}$ is
\[ H_K(q) = \begin{cases} 3q^2\frac{1}{(1-q)^5}, \quad \sigma = (1), \\[5pt]
6q\frac{1+q}{(1-q)^5}, \quad \sigma = (2).
\end{cases}\]
The exponent in $q^2$ for $\sigma = (1)$ is due to the fact that the level of reduction of $N(\sigma^\sharp)$ is $2$. 

Finally, if $\sigma = (3+d)$ for some $d\geq 0$, then the resolution of $L = L(\sigma^\sharp)$ is of length $3$ and the resolving modules $Z_i = Z_i^{\sigma^\sharp}$ are $Z_1 = N(-1,-3,-4-d)$, $Z_2 = N(-2,-4,-4-d)$ and $Z_3 = N(-4,-4,-4-d)$.
For $K=\Ker{\phi_\sigma}$, formula (\ref{hsformula}) yields
\[H_K(q) =  \frac{3}{2}q  \frac{(d+5)(d+2)+2(d+3)q}{(1-q)^5}. \]

\section{Boundary non-stable cases and injectivity of $\phi$}
The method yields some immediate interesting results.

\begin{prop}\label{eventhm}

Let $k \geq 2$ and $n=2k-2$. Then
\[ \textnormal{Ker } \phi \cong \bigoplus_{\sigma} E_\sigma \otimes  \textnormal{Ker } \phi_\sigma \cong \bigoplus_{\sigma} E_\sigma \otimes N(\sigma^\sharp - 2\varepsilon_1)\]
    under the action of the Howe dual pair $(O(n),\mathfrak{sp}(2k))$ with $\sigma$ running over all Young diagrams $\sigma \in \Sigma_{n,k}$ of depth $k-1$ satisfying $\sigma_{k-1} \geq 2.$ 
\end{prop}

\begin{proof}
    By case analysis depending on the depth of Young diagrams, we check conditions (i)-(iii) for non-compact positive roots and compute the group $W_\lambda$ and the root system $\Delta_\lambda$ associated to $\lambda = \sigma^\sharp$ for each $\sigma \in \Sigma_{n,k}$. It turns out that for all depths other than $k-1$, the group $W_\lambda$ is trivial and hence the resolution of $L(\lambda)$ from Theorem \ref{resolution} has length $r_\lambda = 0$. In these cases, $N(\lambda) \cong L(\lambda)$ is irreducible and $\Ker{\phi_\sigma}=0$. When the depth of $\sigma$ equals $k-1$, the group $W_\lambda$ depends on the value of $\sigma_{k-1}$. If $\sigma_{k-1}\geq 2$, then $W_\lambda$ has two elements and the associated root system is $\Delta_\lambda = \{\pm2 \varepsilon_1\}$. Therefore $r_\lambda =1$ and $\Ker{\phi_\sigma}$ is isomorphic to $Z_1^\lambda = N(\lambda - 2\varepsilon_1)$. For $\sigma_{k-1} < 2$, the group $W_\lambda$ is again trivial and $\Ker{\phi_\sigma}=0$.  
\end{proof}

\begin{prop} \label{oddthm}
Let $k \geq 3$ and $n=2k-3$. Then
\[ \textnormal{Ker } \phi \cong \bigoplus_{\sigma} E_\sigma \otimes  \textnormal{Ker } \phi_\sigma\]
    under the action of the Howe dual pair $(O(n),\mathfrak{sp}(2k))$ with $\sigma$ running over all Young diagrams $\sigma \in \Sigma_{n,k}$ of depths $k-2$ or $k-1$ satisfying $\sigma_{k-2} \geq 2$ and
  \[ \textnormal{Ker }\phi_\sigma \cong
 \begin{cases}
  N(\sigma^\sharp - 2\varepsilon_1 - 2\varepsilon_2), \quad  \textnormal{if depth}(\sigma) = k-2,   \\
   N(\sigma^\sharp - \varepsilon_1 - \varepsilon_2), \quad  \textnormal{if depth}(\sigma) = k-1.
 \end{cases} \]   
\end{prop}

\begin{proof}
    The proof strategy is the same as in the even-dimensional situation. One only needs to show that $r_{\sigma^\sharp}=1$ when $\sigma$ is of depth $k-2$ or $k-1$ with $\sigma_{k-2}\geq 2$ and that $r_{\sigma^\sharp}=0$ otherwise. In the first case, having computed the root system $\Delta_{\sigma^\sharp}$, we use its only positive root $\gamma = \varepsilon_1 + \varepsilon_2$ to calculate the module $Z_1^{\sigma^\sharp} = N((s_\gamma(\sigma^\sharp+\rho))^+-\rho)$ isomorphic to $\Ker{\phi_\sigma}$.
\end{proof}

As a very crude consequence, the kernel of $\phi$ is non-zero in the boundary non-stable cases $n=2k-2$ and $n=2k-3$ with $n \geq 2$. Together with the following lemma, this fact yields the first main result of this paper. Since we have to consider changing numbers of variables, we introduce notation $\phi^{(n,k)}$ for the map 
\[\phi^{(n,k)}: \inv(M_{k\times n})\otimes \har(M_{k \times n}) \rightarrow \pol(M_{k \times n}), \quad I \otimes H \mapsto IH, \]
which carries out separation of $k$ variables of dimension $n$. The lemma, interesting in its own right, says that in a fixed dimension, the kernels form an ascending chain of sets as we increase the number of variables.

\begin{lemma} \label{embedding}
 Let $n \in \N^+$ be fixed and let $k\leq l$. Then there is an embedding $\textnormal{Ker }\phi^{(n,k)} \rightarrow \textnormal{Ker }\phi^{(n,l)} $.
 \end{lemma}

 \begin{proof}
The claim follows from the three inclusions $\pol(M_{k \times n}) \subset \pol(M_{l \times n})$, $\har(M_{k \times n}) \subset \har(M_{k \times n})$ and $\inv(M_{k \times n}) \subset \inv(M_{l \times n})$. If these hold, then the kernels are also in inclusion since the map $\phi^{(n,k)}$ becomes the restriction of $\phi^{(n,l)}$ to the subset $\inv(M_{k \times n}) \otimes \har(M_{k \times n})$.

The first inclusion is clear, while the second follows from the definition of Laplace operators $\Delta_{ij}$. Indeed, harmonics on $M_{k \times n}$ remain harmonic on $M_{l \times n}$ since all new Laplacians $\Delta_{ij}$ with $k < j \leq l$ involve differentiation with respect to a new variable $x_{jm}$ for some $m=1,\dots,n$ in each term.

The final inclusion follows from the SFT. The ring $\inv(M_{k \times n})$ is the quotient of the polynomial ring $\mathcal{R}_k = \C[r_{11},\dots,r_{kk}]$ in variables $r_{ij}$ with $1 \leq i \leq j \leq k$ by the ideal $\mathcal{J}_k$ generated by all $(n+1)\times (n+1)$ minors of the symmetric $k \times k$ matrix $(r_{ij})$. Given the obvious inclusion $\mathcal{R}_k \subset \mathcal{R}_l$, one only has to check that $\mathcal{J}_k = \mathcal{J}_l \cap \mathcal{R}_k$, which ensures injectivity of the ring map
\[\inv(M_{k\times n}) = \mathcal{R}_k/\mathcal{J}_k \rightarrow \mathcal{R}_l/\mathcal{J}_l = \inv(M_{l\times n}), \quad r + \mathcal{J}_k \mapsto r + \mathcal{J}_l.\]
The equality of ideals $\mathcal{J}_k = \mathcal{J}_l \cap \mathcal{R}_k$ is an elementary algebraic exercise.
\end{proof}

  \begin{thm} \label{mainthm}
     The map $\phi^{(n,k)}: \inv \otimes \har \rightarrow \pol(M_{k \times n})$ is injective if and only if $n \geq 2k-1$. 
 \end{thm}

 \begin{proof}
 If $n=1$, then the theorem holds. Indeed, if $k=1$, then $\C[x] \cong \C[x^2] \otimes (\C \oplus x\C) = \inv \otimes  \har$, where $x$ is a scalar variable. If $k \geq 2$, then $x^2\otimes y - xy  \otimes x$ is a non-zero element of Ker $\phi$, where $x,y$ are any two of the $k$ scalar variables. 
 
     Fix $n\geq 2$ and consider $k$ such that $n=2k-2$ if $n$ is even and $n=2k-3$ if $n$ is odd. Then $k$ is the least number of variables such that $n<2k-1$ and for all $l \geq k$, it holds that $\Ker{\phi^{(n,l)}} \neq 0$ by Lemma \ref{embedding}. Hence if $n < 2l-1$, the map $\phi$ is not injective. 
     
     For the converse, see \cite{Lav}. In fact, the proof can be quickly replicated with the terminology at hand. It is sufficient to check that all non-compact positive roots fail to satisfy condition (i) when $n\geq 2k-1$. Then the generalized Verma module $N(\lambda)$ is irreducible for each relevant $\lambda$ and so $\Ker{\phi}=0$.
 \end{proof}

In Section \ref{hwvectors}, we give a constructive proof of this result by producing formulas for non-zero elements of $\Ker{\phi^{(n,k)}}$ whenever $n < 2k-1$.

\section{Generators of the kernel} \label{gen}
We present the second main result of the paper, a general construction of highest weight vectors in the kernel of the multiplication map $\phi^{(n,k)}$ in the full non-stable range $n<2k-1$.

\subsection{Highest weight harmonics} \label{harmonics}

Before describing generators of $\Ker{\phi}$, we need to understand the structure of harmonic polynomials in each dimension. Both \cite{G} and \cite{Ho2} give explicit constructions for highest weight vectors in $\har$, which we now briefly recall and adjust to fit our setup. Fix $n,k\in \N^+$ and let $m=\min(n,k)$. For $j=1,\dots,m$, let $\delta_j$ denote the determinant of the $j\times j$ submatrix forming the bottom right corner of the variable matrix $x=(x_{il})_{k \times n}$, i.e. 
\[ \delta_j = \det{\begin{pmatrix}
x_{k-j+1,n-j+1} & x_{k-j+1,n-j+2} & \dots & x_{k-j+1,n} \\
x_{k-j+2,n-j+1} & x_{k-j+2,n-j+2} & \dots & x_{k-j+2,n} \\
\vdots & \vdots & \vdots & \vdots \\
x_{k,n-j+1} & x_{k,n-j+2} & \dots & x_{k,n}
\end{pmatrix}}.\]
Take $\sigma \in \Sigma_{n,k}$ and let $a_j = \sigma_j - \sigma_{j+1}$ for $j = 1,\dots,m$, where $\sigma_{m+1} := 0$. Then 
\begin{equation} \label{hwharmonic}
    \delta_\sigma = \delta_1^{a_1} \delta_2^{a_2} \dots \delta_m^{a_m} 
\end{equation}
forms a joint $(SO(n),\mathfrak{gl}(k))$ highest weight vector in $\har$ of weight $\sigma^\sharp$ with respect to $\mathfrak{gl}(k)$. Note that the action on the determinant $\delta_j$ by an operator $h_{is}$ with $i\neq s$ exchanges the occurrence of the variable $x_{sl}$ with the occurrence of $x_{il}$ for any $l = n-j+1,\dots,n$.

\subsection{Calculation of highest weights in $\Ker{\phi}$}
The method sketched out in Section \ref{generators} tells us that we should begin with finding the highest weight $\lambda^\prime$ of $Z_1^\lambda = N(\lambda^\prime)$ for each $\sigma \in \Sigma_{n,k}$ and $\lambda = \sigma^\sharp$. We will show that $\lambda^\prime$ is the highest weight of $\Ker{\phi_\sigma}$ unless $\sigma  = 0$ is trivial. Quite surprisingly, it turns out that the task of calculating $\lambda^\prime$ can be radically simplified in two steps. Firstly, we show precisely how the computation depends on the number of variables $k$ in a fixed dimension $n$, which will later allow us to restrict our attention to the case $k=n+1$. Secondly, we explain why it is sufficient to consider only \textit{`narrow'} Young diagrams $\sigma \in  \Sigma_n$ with at most two columns instead of considering all $\sigma \in \Sigma_n$. Let us begin the process by introducing some new notation. Before doing so, recall that $\Sigma_{n,k}$ denotes the set of Young diagrams of depth at most $k$ with at most $n$ boxes in the first two columns. 

Since we have to consider multiple Young diagrams $\sigma \in \Sigma_{n,k}$ as well as different values of $k$ at a time, we denote by $\lambda_{(k)}^\prime(\sigma)$ the weight 
\[ \lambda_{(k)}^\prime(\sigma) = (s_\gamma(\sigma^\sharp + \rho))^+ - \rho\] calculated using the root $\gamma = \min{\Gamma^k_{\sigma^\sharp}}$ from Proposition \ref{propgen}. Here, $\Gamma^k_\lambda$ denotes the set $\Gamma_\lambda$ of positive non-compact roots in $\mathfrak{sp}(2k)$ satisfying conditions (i)-(iii). If $k$ is irrelevant or clear from the context, we simply write $\lambda^\prime(\sigma) = \lambda_{(k)}^\prime(\sigma)$ and $\Gamma_\lambda = \Gamma^k_\lambda$. It may happen that $\lambda^\prime(\sigma)$ is undefined, though, if the set $\Gamma_{\sigma^\sharp}$ is empty. Therefore we only consider $\lambda_{(k)}^\prime(\sigma)$ for $\sigma \in \Sigma_{n,k}^0$, where $\Sigma_{n,k}^0$ denotes
\[ \Sigma_{n,k}^0 = \{\sigma \in \Sigma_{n,k}: \Gamma^k_{\sigma^\sharp} \neq \emptyset \} . \]
Note that by Theorem \ref{resolution}, $\Sigma^0_{n,k}$ is exactly the set of $\sigma \in \Sigma_{n,k}$ for which the generalized Verma module $N(\sigma^\sharp)$ is reducible, i.e. $N(\sigma^\sharp) \ncong L(\sigma^\sharp)$. By the end of this section, we will be able to give an explicit description of $\Sigma_{n,k}^0$ based on the computations of $\lambda^\prime$ below. The following lemma says that the operation $\sigma \mapsto \lambda_{(k)}^\prime(\sigma)$ is independent of $k$ in a natural sense. 

\begin{lemma} \label{L1}
    If $\sigma \in \Sigma_{n,k}^0$, then $\sigma \in \Sigma_{n,k+1}^0$ and
    \[ \lambda^\prime_{(k+1)}(\sigma) = (-\frac{n}{2},\lambda^\prime_{(k)}(\sigma)).\]
    Conversely, if $\sigma \in \Sigma_{n,k+1}^0 \cap \Sigma_{n,k}$ and $\lambda^\prime_{(k+1)}(\sigma) = (-\frac{n}{2},\nu)$ for some $k$-tuple $\nu$, then $\sigma \in \Sigma_{n,k}^0$ and $\lambda^\prime_{(k)}(\sigma)=\nu$.
\end{lemma}

\begin{proof}
    The lemmas of this section are technical results which require checking conditions (i)-(iii) in quite a general setting. The procedure is carried out explicitly in \cite[Section 7]{CZ}, where one can find a case-by-case description of the set $\Gamma_{\sigma^\sharp}$ of positive non-compact roots satisfying (i)-(iii).
    
    Let us fix $\sigma \in \Sigma_{n,k}$ and write $\Gamma^m = \Gamma^m_{\sigma^\sharp}$ for $m=k,k+1$. To prove the first claim, it is sufficient to show that if $\varepsilon_i + \varepsilon_j$ is the lowest root in $\Gamma^k$, then $\varepsilon_{i+1}+\varepsilon_{j+1}$ is the lowest root in $\Gamma^{k+1}$. Indeed, calculations in \cite{CZ} say that there is an index set $I_m \subset \{1,\dots,m\}$ such that 
    \[\Gamma^m = \{ \varepsilon_{m+1-i}+\varepsilon_{m+1-j} : i,j \in I_m \}.  \]
    In most cases, there is the additional requirement for $i,j \in I_m$ to be distinct, but this is not important for us here. What matters is that the set $I_m$ is described as $I_m = \{1,\dots,m\} \setminus J$, where the set $J$ depends only on $n$ and $\sigma$ and not on $m$. Therefore $I_{k+1} = I_k$ or $I_{k+1}= I_k \cup \{k+1\}$, which proves the claim concerning lowest roots of $\Gamma^k$ and $\Gamma^{k+1}$.

    For the converse, it suffices to show that $\Gamma^{k+1} \neq \emptyset$ implies $\Gamma^k \neq \emptyset$ under the assumption that the first entry of $\lambda^\prime_{(k+1)}(\sigma)$ is equal to $-n/2$. The rest follows from what we have proved above. If $I_{k+1}= I_k$, then the claim clearly holds, so let us assume that $I_{k+1}=I_{k} \cup \{k+1\}$ and suppose that $\Gamma^k = \emptyset$ for contradiction. Then the lowest root in $\Gamma^{k+1}$ is of the form $\gamma = \varepsilon_1 + \varepsilon_j$ for some $j\geq 1$. Using the fact that the entries of $\sigma^\sharp + \rho$ form a strictly decreasing sequence, we deduce that the first entry of $\lambda^\prime_{(k+1)}(\sigma) = (s_\gamma(\sigma^\sharp + \rho))^+ - \rho$ must be bounded by $-n/2-\sigma_{k-1} - 1$ from above, a contradiction.
\end{proof}

Let us denote by $\Omega_{n,k}\subset \Sigma_{n,k}$ the Young diagrams of $\Sigma_{n,k}$ with at most two columns. An element 
\[\sigma = (\overbrace{\underbrace{2,\dots,2}_t,1,\dots,1}^d) \in \Omega_{n,k}\] is uniquely determined by the non-negative integers $t$ and $d=$ depth$(\sigma)$, which satisfy $t + d \leq n$ and  $t\leq d \leq k$. We call such $\sigma$ the \textit{narrow Young diagram of type} $(t,d)$. Also define $\Omega^0_{n,k} = \Omega_{n,k} \cap \Sigma^0_{n,k}$, the set of narrow Young diagrams $\sigma$ for which $\lambda_{(k)}^\prime(\sigma)$ exists. For example, Proposition \ref{eventhm} describing the even boundary non-stable case $n=2k-2$ says that the set $\Omega^0_{2k-2,k}$ only contains the narrow Young diagram of type $(k-1,k-1)$. There is a nice formula for $\lambda^\prime(\sigma)$ when $\sigma$ is narrow.

\begin{lemma} \label{L2}
    For any non-trivial narrow Young diagram
    \[ \sigma = (\underbrace{2,\dots,2}_t,\underbrace{1,\dots,1}_b) \in \Omega_{n,n+1}, \]it holds that $\sigma \in \Omega_{n,n+1}^0$ and 
    \[\lambda_{(n+1)}^\prime(\sigma) = (\underbrace{-\frac{n}{2},\dots,-\frac{n}{2}}_t,\underbrace{-\frac{n}{2}-1,\dots,-\frac{n}{2}-1}_b, -\frac{n}{2}-2,\dots,-\frac{n}{2}).\]
\end{lemma}

\begin{proof}
    Except for certain boundary cases, this is the content of the concluding example in \cite[p.~374]{EW} combined with Lemma \ref{L1}. The boundary cases are those in which $t=0$ or $b=0$ and checking conditions (i)-(iii) to compute the set $\Gamma_{\sigma^\sharp}$ and thus the weight $\lambda^\prime(\sigma)$ is fairly easy.
\end{proof}

The final technical lemma enables the reduction from all Young diagrams to the narrow ones. Thanks to Lemma \ref{L1}, it is sufficient to consider the case $k=n$.

\begin{lemma} \label{L3}
    Let $\sigma \in \Omega_{n,n}$ be of type $(t,d)$ and let $\mu = (\mu_1,\dots,\mu_t)$ be a partition of depth at most $t$. If $t>0$, then $\sigma \in \Omega_{n,n}^0$ and $\sigma + \mu \in \Sigma_{n,n}^0$. In this case,   \[ \lambda_{(n)}^\prime(\sigma + \mu) = \lambda_{(n)}^\prime(\sigma) + (0,\dots,0,-\mu_t,-\mu_{t-1},\dots,-\mu_1).\]
\end{lemma}

\begin{proof}
 The claim that $t>0$ implies $\sigma \in \Omega^0_{n,n}$ is a direct consequence of Lemmas \ref{L1} and \ref{L2}. As in Lemma \ref{L1}, the proof relies on the explicit description of $\Gamma_\lambda$ provided by \cite{CZ}. The proof can be divided into two steps. In the first step, we need to show that 
    \begin{equation} \label{minima}
        \min{\Gamma_{\sigma^\sharp}} = \min{\Gamma_{(\sigma+\mu)^\sharp}},
    \end{equation} which immediately proves that $\sigma + \mu \in \Sigma_{n,n}^0$. 
    In the second step, we show that the operation $(s_\gamma(\sigma^\sharp+\rho))^+$ only affects the first $n+1-t$ indices of $\sigma^\sharp + \rho$, where  $\gamma$ is the minimum from (\ref{minima}). The proof of this claim is a straightforward case-by-case analysis, so we skip it and prove only equality (\ref{minima}).
    
    Let us denote by $I$ the index set of $\Gamma_{\sigma^\sharp}$ and by $J$ the index set of $\Gamma_{(\sigma+\mu)^\sharp}$, i.e. $\Gamma_{\sigma^\sharp} = \{ \varepsilon_{n+1-i} + \varepsilon_{n+1-j} : i,j \in I\}$ and similarly for $\Gamma_{(\sigma+\mu)^\sharp}$ and $J$. In most cases, we also require $i,j$ to be distinct. Either way, to prove equality (\ref{minima}), it is sufficient to show that $I\subset J$ and $j > \max{I}$ for all $j\in J\setminus I$. This is done by case analysis depending on the parity of $n$ and the depth $d$ of $\sigma$. For example, if $n=2p$ is even and $d<p$, then the calculations in \cite{CZ} say that $I = \{n-d,n-t+1\}$ and $I \subset J \subset \{ n-d,n-t+1,n-t+2,\dots,n\}$, so the claim holds. The remaining cases are very similar. 
\end{proof}

Note that any $\sigma \in \Sigma_{n,k} \setminus \Omega_{n,k}$ can be uniquely decomposed as $\sigma_0 + \mu$ with a `maximal' $\sigma_0 \in \Omega_{n,k}$ and a suitable partition $\mu$, so Lemma \ref{L3} is without loss of generality. Indeed, let $t=\max{\{i:\sigma_i \geq 2\}}$ and $d=$ depth$(\sigma)$. Then we can take the narrow Young diagram $\sigma_0 \in \Omega_{n,k}$ of type $(t,d)$ and put $\mu_i = \sigma_i - 2$ for $i=1,\dots,t$. Moreover, if $\sigma \in \Sigma_{n,k}$ has such decomposition $\sigma = \sigma_0 + \mu$ and $\sigma \in \Sigma^0_{n,k}$, then $\sigma_0 \in \Omega_{n,k}^0$. This follows for $k=n$ from Lemma \ref{L3} and then for general $k$ from Lemma \ref{L1}.

A consequence of Lemmas \ref{L1}-\ref{L3} is an explicit description of $\Sigma_{n,k}^0$ advertised above. Of course, the set $\Sigma_{n,k}^0$ is always empty in the semistable range $n\geq 2k-1$.

\begin{prop}
Assume that $n,k \in \N^+$ satisfy $n <2k-1$. 
\begin{enumerate}
    \item If $k>n$, then $\Sigma_{n,k}^0 = \Sigma_{n,k} = \Sigma_n$.
    \item If $k \leq n$, then 
    \[\Sigma_{n,k}^0 = \{\sigma \in \Sigma_{n,k} : \sigma_{n+1-k} \geq 2 \}. \]
\end{enumerate}

\end{prop}

\subsection{Highest weight vectors in $\Ker{\phi}$} \label{hwvectors} 
Having described $\lambda_{(k)}^\prime(\sigma)$ for each $\sigma \in \Sigma_{n,k}^0$, we are finally prepared to find a corresponding relation in $\Ker{\phi_\sigma}$ of weight $\lambda_{(k)}^\prime(\sigma)$, which generates $\Ker{\phi_\sigma}$ as an $\mathfrak{sp}(2k)$-module. Each generating relation in dimension $n$ will be formed by a minor of order $n+1$ in a matrix of rank $n$. The results of the previous section allow us to consider only the case $k=n+1$ and $\sigma \in \Omega^0_{n,n+1}$, from which it is easy to derive generators of $\Ker{\phi^{(n,k)}}$ in the full non-stable range $n<2k-1$. 

Recall the bilinear form $(\cdot,\cdot)$ on $\C^n$ given in equation (\ref{symform}). Let $J_n$ denote the matrix of $(\cdot,\cdot)$, i.e.
\[J_n =
\begin{pmatrix}
    0 & 0 & \dots & 0 & 1 \\
    0 & 0 & \dots & 1 & 0 \\
    \vdots & \vdots & \vdots & \vdots & \vdots  \\
    0 & 1 & \dots & 0 & 0 \\
    1 & 0 & \dots & 0 & 0
\end{pmatrix}.  \] 
For any $n,k \in \N^+$, we define $M^{(n,k)}$ to be the symmetric $(n+k)\times (n+k)$ matrix
\[ M^{(n,k)} = \begin{pmatrix}
    x \\
    I_n
\end{pmatrix}J_n (x^T, I_n)
=   \begin{pmatrix}
r_{11} & \dots & r_{1k} & x_{1n} & \dots & x_{11} \\
       \vdots & \vdots & \vdots & \vdots & \vdots & \vdots \\
       r_{k1} & \dots & r_{kk} & x_{kn} & \dots & x_{k1} \\
       x_{1n} & \dots & x_{kn} & 0 & \dots & 1 \\
       \vdots & \vdots & \vdots & \vdots & \vdots & \vdots \\
       x_{11} & \dots & x_{k1} & 1 & \dots & 0
\end{pmatrix}
,\]
where $I_n$ is the $n\times n$ identity matrix. The rank of $M^{(n,k)}$ is clearly $n$ and hence all its minors of order $n+1$ are zero. Let $a,b > 0$ and define $M_{a,b} = M_{a,b}^{(n,k)}$ as the square submatrix of $M^{(n,k)} =(m_{ij})$ of order $n+1$ whose entry at position $(i,j)$ is equal to $m_{a+i-1,b+j-1}$. Then all relations in $\Ker{\phi^{(n,k)}}$ emerge from the minors $\det{M_{a,b}}$ with suitable choices of $a,b$. Here, we view $\det{M_{a,b}}$ as an element of $\inv \otimes \har$, i.e. we put the tensor product sign $\otimes$ between each product of invariants $r_{ij}$ and each product of variables $x_{ij}$. The expression $\det{M_{a,b}}$ is a non-zero element of $\inv \otimes \har$ for suitable $a$ and $b$, while $\phi(\det{M_{a,b}})$ is the zero polynomial. 

\begin{thm}
    Suppose that  \[\sigma = (2,\dots,2,1,\dots,1) \in \Omega^0_{n,n+1} \] is of type $(t,d)$ with $d>0$. Then \[\det{M^{(n,n+1)}_{d+1,t+1}} = \begin{vmatrix}
       r_{d+1,t+1} & \dots & r_{d+1,n+1} & x_{d+1,n} & \dots & x_{d+1,n+1-t} \\
       \vdots & \vdots & \vdots & \vdots & \vdots & \vdots \\
       r_{n+1,t+1} & \dots & r_{n+1,n+1} & x_{n+1,n} & \dots & x_{n+1,n+1-t} \\
       x_{t+1,n} & \dots & x_{n+1,n} & 0 & \dots & 0 \\
       \vdots & \vdots & \vdots & \vdots & \vdots & \vdots \\
       x_{t+1,n+1-d} & \dots & x_{n+1,n+1-d} & 0 & \dots & 0
    \end{vmatrix}\] 
    is a highest weight vector in $\Ker{\phi^{(n,n+1)}_\sigma}$ of weight $\lambda_{(n+1)}^\prime(\sigma)$.
\end{thm}

\begin{proof}
    Let us divide $M = M^{(n,n+1)}_{d+1,t+1}$ into the natural blocks:
    \[ M = \begin{pmatrix}
        R & Y \\
        Z^T & 0_{d\times t}
    \end{pmatrix}, \; R = (r_{ij})^{i \geq d+1}_{j \geq t+1}, \; Y = (x_{ij})^{i \geq d+1}_{j \geq n+1-t}, \; Z = (x_{ij})^{i \geq t+1}_{j \geq n+1-d}.  \]
    Strictly speaking, the matrices $Y$ and $Z$ are equal to the corresponding variable submatrices of $(x_{ij})$ only after reversing their columns. However, we only consider their minors and any potential change of sign does not affect the argument. 
    
    Using cofactor expansions in multiple rows and columns, we observe that $D_\sigma = \det{M_{d+1,t+1}}$ splits into a sum
    \[ D_\sigma = \sum_{I,J} \pm D_{I,J}, \quad D_{I,J} =   (\det{R^I_J}) \otimes (\det{Y_I})  (\det{Z_J}), \]
    where $I$ runs over $t$-element subsets of $\{d+1,d+2,\dots,n+1\}$ and $J$ runs over $d$-element subsets of $\{t+1,t+2,\dots,n\}$. Moreover, $\det{Y_I}$ denotes the $t\times t$ minor of $Y$ on rows indexed by $I$, $\det{Z_J}$ denotes the $d \times d$ minor of $Z$ on rows indexed by $J$, and $\det{R^I_J}$ denotes the minor of $R$ of order $n+1-d-t$ given by deleting rows in $I$ and columns in $J$. 
   
    As a prerequisite, notice that each term $D_{I,J}$ in $D_\sigma = \det{M_{d+1,t+1}}$ is an element of $\inv \otimes F_{\sigma^\sharp}$ by the description of harmonics in Section \ref{harmonics} and hence $D_\sigma \in \Ker{\phi_\sigma}$. We should also check that $D_\sigma \neq 0$. This follows from the fact that monomials in $r_{ij}$ of degree $n+1-d-t$ form a linearly independent set in $\inv$, which is a consequence of the SFT since $d > 0$. We can fix any $t$-element subset  $I \subset \{d+1,d+2,\dots,n+1\}$ and look at the non-zero term $T_I$ in $D_\sigma$ of the form
    \[ T_I = \pm  (\prod_{i \in I} r_{ii}) \otimes (\det{Y_I})(\det{Z_{K \cup I}}),\]
    where $K = \{t+1,t+2,\dots,d\}$. Then there are no other terms in $D_\sigma$ having $\prod_{i \in I} r_{ii}$ as the left factor. Hence the non-zero term $T_I$ cannot get cancelled out by any of the remaining terms in $D_\sigma$ and so $D_\sigma \neq 0$.
    
    To prove the theorem, it is now sufficient to check that for each $I,J$ it holds that
    \[ h_{ii} \cdot D_{I,J} = \begin{cases}
        \frac{n}{2} D_{I,J}, \quad \text{if} \quad i = 1,\dots,t, \\
        (1+\frac{n}{2})D_{I,J}, \quad \text{if} \quad i = t+1,\dots,d, \\
        (2+\frac{n}{2})D_{I,J}, \quad \text{if} \quad i = d+1,\dots,n+1.
    \end{cases}\]
     This is quite straightforward since the scalar action of $h_{ii}$ on $\det{Y_I}\det{Z_J} \in F_{\sigma^\sharp}$ is captured in $\sigma^\sharp$, while the scalar action of $h_{ii}$ on a product $\prod_{a,b}r_{ab}$ is equal to the sum $\Sigma_{a,b}(-\varepsilon_a - \varepsilon_b)$ of the corresponding non-compact roots evaluated on $h_{ii}$. In particular, the eigenvalue of $h_{ii}$ acting on $\det{R^I_J}$ is equal to $0$ if $i \in I \cap J$, it is equal to $2$ if $i \notin I \cup J$ and it equals $1$ otherwise. 
\end{proof}

The lemmas of the previous section enable the following quick generalization to any number of variables and any Young diagram. Recall the definition (\ref{hwharmonic}) of the highest weight harmonic $\delta_\sigma \in F_{\sigma^\sharp}\subset \har(M_{k \times n})$ for $\sigma \in \Sigma_{n,k}$.

\begin{cor} \label{maincor}
    Let $\sigma_0 \in \Omega_{n,k}^0$ be of type $(t,d)$ and let $\mu = (\mu_1,\dots,\mu_t)$ be a partition. Put $\sigma = \sigma_0 + \mu$. Then $\sigma \in \Sigma^0_{n,k}$ and
    \[ (\det{M_{d-n+k,t-n+k}})\delta_\mu \]
    is a highest weight vector in $\Ker{\phi_\sigma^{(n,k)}}$ of weight $\lambda_{(k)}^\prime(\sigma)$.
\end{cor}

\begin{proof}
    First, put $k=n$ and $\mu = 0$ and apply Lemma \ref{L1} to obtain the result of the theorem for narrow Young diagrams in the case of $n$ variables. Then for any $\mu$ apply Lemma \ref{L3}. Finally, apply Lemma \ref{L1} again to generalize to any $k$ in the non-stable range $n<2k-1$.
\end{proof}

As a very special case, we obtain a constructive proof of Theorem \ref{mainthm}. Indeed, there is a non-zero element $\det{M^{(n,k)}_{1,1}}$ in $\Ker{\phi^{(n,k)}}$ in the boundary non-stable cases $n=2k-2$ or $n=2k-3$, so $\Ker{\phi^{(n,k)}} \neq 0$ whenever $n<2k-1$ by Lemma \ref{embedding}.

\subsection{Concluding remarks}
Let us finish with a brief discussion of what is known and what can be done further. We also include a small, but interesting result in Corollary \ref{corollary}.

\begin{rem}
    A result analogous to Corollary \ref{maincor} was presented by R. Howe during his lectures \cite{Ho3} delivered in Prague in 2015, where he described generators of the kernel of a multiplication map with underlying symmetry given by the group $GL(n,\C)$. The author was unable to find a published version of the result. The theoretical apparatus mentioned in the lectures is presented in \cite{HKL}.
\end{rem}

\begin{rem}
  For each $\sigma \in \Sigma_{n,k}^0$, we can act on the highest weight vector $D_\sigma^1 =  (\det{M_{d-n+k,t-n+k}})\delta_\mu$ from Corollary \ref{maincor} by operators $h_{ij}$ with $i < j$ until we obtain enough linearly independent relations $D_\sigma^1,\dots,D_\sigma^s \in \Ker{\phi_\sigma}$, i.e. with $s=\dim{F_{\lambda^\prime}}$ and $\lambda^\prime = \lambda^\prime_{(k)}(\sigma)$. Then it follows that $\Ker{\phi_\sigma}$ is a quotient of the free $\inv$-module on $D_\sigma^1,\dots,D_\sigma^s$. This set is moreover a minimal set of $\inv$-module generators of the kernel. For an explicit construction of the basis $D_\sigma^1,\dots,D_\sigma^s$ from the highest weight vector $D_\sigma^1$, we refer to \cite{MY}.   
\end{rem}

By exhibiting non-zero elements of $\Ker{\phi_\sigma^{(n,k)}}$ for each non-trivial $\sigma \in \Sigma_{n,k}^0$, we have proved the difficult implication in the following criterion for irreducibility of generalized Verma modules $N_\A(\lambda)$ for the algebra $\A$ of $O(n)$-invariant differential operators on $\pol(M_{k\times n})$. The easy claim is that irreducibility of $N(\lambda)$ implies irreducibility of $N_\A(\lambda)$ as in this case both must be isomorphic to $L(\lambda)$.

\begin{cor} \label{corollary}
  Let $\sigma \in \Sigma_{n,k}$ be a non-trivial Young diagram. It holds that $\sigma \in \Sigma^0_{n,k}$ if and only if $\Ker{\phi_\sigma^{(n,k)}} \neq 0$. In other words, the generalized Verma module $N(\sigma^\sharp)$ is reducible if and only if the generalized Verma module $N_\A(\sigma^\sharp)$ for $\A$ is reducible.
\end{cor}

 In Corollary \ref{maincor}, we gave a list of generators of the kernel of $\phi$. It remains an open problem to understand relations among these generators. More precisely, we would like to understand the $\inv$-module structure of each $\Ker{\phi_\sigma}$. Propositions \ref{eventhm} and \ref{oddthm} describe this structure when $n=2k-2$ or $n=2k-3$. One can go further and calculate the resolutions of the kernel by generalized Verma modules in the range $k \leq n < 2k-3$. In the remaining cases $n<k$, it is unclear what the structure of $\Ker{\phi}$ might look like. To finish on a more positive note, let us include an example of a situation in which we have all the information.

\begin{ex}
   Consider the case of dimension $n=2$ with $k=2$ variables. We have $\Omega^0_{2,2}=\{(2)\}$ and hence the only minor of $M^{(2,2)}$ to consider is the principal $3\times 3$ minor
    \[ D_{(2)} = \det{\begin{pmatrix}
    r_{11} & r_{12} & x_{12} \\
    r_{12} & r_{22} & x_{22} \\
    x_{12} & x_{22} & 0
\end{pmatrix}} = 2r_{12}\otimes x_{12}x_{22} - r_{11}\otimes x_{22}^2 - r_{22}\otimes x_{12}^2.\]
For $\sigma  = (a +2) \in \Sigma^0_{2,2}$ with $a\geq 0$, we multiply by $x_{22}^a$ from the right to obtain
\[ D_\sigma = D_{(2)}x_{22}^a\] as the highest weight vector in $\Ker{\phi_\sigma}$ of weight $\lambda^\prime(\sigma)=(-3,-3-a)$. Acting by $h_{12}$ on $D_\sigma$ and using Proposition \ref{eventhm}, we conclude that $\Ker{\phi_\sigma}$ is the free $\inv$-module generated by 
\[ D_{(2)}x_{12}^i x_{22}^{a-i}, \quad i = 0,1,\dots,a.\]
\end{ex}

\subsection*{Acknowledgments}
The author is sincerely grateful for guidance, advice and valuable feedback from Roman L\' avi\v cka.

\def\bibname{Bibliography}


\begin{thebibliography}{99}
	\addcontentsline{toc}{chapter}{\bibname}

\bibitem{B} D. Be\v{d}at\v s: Howe duality and invariant differential equations. Master thesis, Charles University, Faculty of Mathematics and Physics, Mathematical Institute, Praha, Czech Republic (2023).

\bibitem{CZ} S.J. Cheng, R.B. Zhang: Howe duality and combinatorial character formula for orthosymplectic Lie superalgebras, Advances in Mathematics, 182, 1 (2004), pp. 124-172. https://doi.org/10.1016/S0001-8708(03)00076-8

\bibitem{DES} M.G. Davidson, T.J. Enright, R.J. Stanke: Differential operators and highest weight representations, Memoirs of the AMS, 94, 455 (1991).

\bibitem{D} J. Dixmier: Enveloping algebras, North-Holland Publishing Company, Amsterdam (1977).

\bibitem{EW} T.J. Enright, J.F. Willenbring: Hilbert series, Howe duality and branching for classical groups, Ann. Math. (2), 159, 1 (2004), pp. 337-375. https://doi.org/10.4007/annals.2004.159.337

 \bibitem{G} R.\ Goodman: Multiplicity-free spaces and Schur-Weyl-Howe duality,  in Representations of Real and p-adic Groups (ed. E-C Tan and C-B Zhu),
	Lecture Note Series--Institute for Mathematical Sciences, Vol. 2,
	World Scientific, Singapore (2004), pp. 304-415.

\bibitem{GW} R. Goodman, N.R. Wallach: Symmetry, Representations, and Invariants, Springer-Verlag, New York (2009).


\bibitem{Ho1} R. Howe: Remarks on classical invariant theory, Trans. AMS, 313, 2 (1989), pp. 539-570. https://doi.org/10.2307/2001418

\bibitem{Ho2} R. Howe: Perspectives on invariant theory: Schur duality, multiplicity-free actions and beyond, The Schur lectures, 1992, Tel Aviv (1995), pp. 1-182.

\bibitem{Ho3} R. Howe: Four stages of Classical Invariant Theory [Lecture notes], Prague (2015). 

\bibitem{HKL} R. Howe, S. Kim, S.T. Lee: Standard Monomial Theory for Harmonics in Classical Invariant Theory, in Representation Theory, Number Theory, and Invariant Theory, Progress in Mathematics 323, Birkhäuser, Cham, Germany (2017), pp. 265-302.


 \bibitem{Hum} J.E. Humphreys: Representations of Semisimple Lie Algebras in the BGG Category $\mathcal O$,
 AMS, Providence, RI (2008).

 \bibitem{Lav} R. L\'avi\v cka: Separation of {V}ariables in the {S}emistable {R}ange, in Topics in Clifford Analysis: Special Volume in Honor of Wolfgang Spr{\"o}{\ss}ig, Springer International Publishing, Cham, Germany (2019), pp. 395-403. 

 \bibitem{MY}  A. Molev, O. Yakimova: Monomial Bases and Branching Rules, Transformation Groups 26 (2021), pp. 995–1024. https://doi.org/10.1007/s00031-020-09585-1

 \bibitem{MW} I.M. Musson, J.F. Willenbring: Invariant Differential Operators and FCR Factors of Enveloping Algebras, Communications in Algebra, 6:10 (2008), pp. 3759-3777. https://doi.org/10.1080/00927870802160529

 \bibitem{S} G.W. Schwarz: Finite-dimensional representations
of invariant differential operators, Journal of Algebra 258 (2002), pp. 160-204. https://doi.org/10.1016/S0021-8693(03)00281-3


\end{thebibliography}
\end{document}